\documentclass[10pt]{article}
\usepackage{amsmath, amsthm, amssymb, amsfonts, mathrsfs, mathtools}
\usepackage{graphicx}
\usepackage{makeidx}
\usepackage[left=2cm,right=2cm,top=2cm,bottom=2cm]{geometry}

\usepackage{caption}
\usepackage{subcaption}

\usepackage{tikz}
\usetikzlibrary{decorations.pathreplacing}
\tikzstyle{vertex}=[auto=left,circle,draw=black,fill=white, inner sep=1.5]
\usepackage{enumitem}
\usepackage{makecell}
\usepackage{blindtext}

\usepackage{hyperref}
\hypersetup{colorlinks=true, citecolor={blue}, linkcolor=blue, filecolor=magenta, urlcolor=cyan}

\newtheorem{theorem}{Theorem}

\newtheorem{lem}[theorem]{Lemma}

\newtheorem{prop}{Proposition}

\newtheorem{definition}{Definition}

\newtheorem{conjecture}{Conjecture}

\title{On the \emph{rna} number of powers of cycles}

\author{ Deepak Sehrawat$^{\ast}$, Anil Kumar, Sweta Ahlawat\\
Department of Mathematics\\
Pandit Neki Ram Sharma Government College Rohtak\\
Rohtak, India - 124001\\
$^{\ast}$Email: deepakssehrawat5@gmail.com
}
\date{}

\begin{document}
\maketitle

\vspace{-0.3in}

\vspace*{0.3in}
\noindent
\textbf{Abstract.} A signed graph $(G,\sigma)$ on $n$ vertices is called a \textit{parity signed graph} if there is a bijective mapping $f \colon V(G) \rightarrow \{1,\ldots,n\}$ such that $f(u)$ and $f(v)$ have same parity if $\sigma(uv)=1$, and opposite parities if $\sigma(uv)=-1$ for each edge $uv$ in $G$. The \emph{rna} number $\sigma^{-}(G)$ of $G$ is the least number of negative edges among all possible parity signed graphs over $G$. In other words, $\sigma^{-}(G)$ is the smallest size of an edge-cut of $G$ such that the sizes of two sides differ at most one.

Let $C_n^{d}$ be the $d\text{th}$ power of a cycle of order $n$. Recently, Acharya, Kureethara and Zaslavsky proved that the \emph{rna} number of a cycle $C_n$ on $n$ vertices is $2$. In this paper, we show for $2 \leq d < \lfloor \frac{n}{2} \rfloor$ that $2d \leq \sigma^{-}(C_n^{d}) \leq d(d+1)$. Moreover, we prove that the graphs $C_n^{2}$ and $C_n^{3}$ achieve the upper bound of $d(d+1)$.

\vspace*{0.15in}

\vspace*{0.1in}

\section{Introduction}\label{intro}

In this paper, all the graphs and signed graphs are considered to be simple, connected and undirected. For any graph theoretic term that is used but not defined in this paper, we refer the reader to ~\cite{Bondy}. If $G$ is a graph, then we denote the set of vertices and the set of edges of $G$ by $V(G)$ and $E(G)$, respectively, and the distance between vertices $u$ and $v$ in $G$ by $\text{dist}_{G}(u,v)$.

A \textit{signed graph} $(G,\sigma)$ consists of a graph $G$ together with a function $\sigma \colon E(G) \rightarrow \{1, -1\}$. In $(G,\sigma)$, graph $G$ is called the \textit{underlying graph} of $(G,\sigma)$ and $\sigma$ a \textit{signature} of $(G,\sigma)$. We say an edge $e$ in $(G,\sigma)$ is \textit{positive} if $\sigma(e)=1$, and \textit{negative} otherwise. By $(G,+)$, we denote a signed graph $(G,\sigma)$ such that $\sigma(e)=1$ for each $e \in E(G)$. 

A special type of signed graph known as \textit{parity signed graph} (see Definition~\ref{def of parity labeling and parity signed graph}) introduced by Acharya and Kureethara in \cite{Acharya1}. Acharya \textit{et al.}~\cite{Acharya2} further characterized some families of parity signed graphs. Parity labeling of a graph is equivalent to a partition of the vertex set of a graph into two subsets $A$ and $B$ such that $||A|-|B|| \leq 1$. Such vertex partition is known as a \textit{Harary partition}.

The \emph{rna} number of a graph $G$, denoted $\sigma^{-}(G)$, is the least number of negative egdes among all the possible parity signed graphs over $G$. In other words, it is the least size of a cut whose sides are nearly equal. The concept of the \emph{rna} number of a parity signed graph was also introduced by Acharya and Kureethara~\cite{Acharya1}. Due to its definition, it is more reasonable to say the \emph{rna} number of a graph.  The \emph{rna} numbers of some families of graphs such as stars, wheels, paths, cycles and complete graphs are computed in~\cite{Acharya1, Acharya2}. Kang et al. \cite{Kang} proved that $\sigma^{-}(G) \leq \lfloor \frac{2m+n}{4} \rfloor$, where $G$ is a graph on $n$ vertices with $m$ edges. Further, they characterized all the parity signed graphs achieving the bound $\lfloor \frac{2m+n}{4} \rfloor$. They also solved some open problems concerning the \emph{rna} number. Very recently, Sehrawat and Bhattacharjya~\cite{Sehrawat} studied the \emph{rna} number for the class of generalized Petersen graphs. 

The \textit{dth power} of a simple graph $G$ is the graph $G^d$ whose vertex set is $V(G)$ and two distinct vertices $u,v$ are adjacent in $G^d$ if and only if $1 \leq \text{dist}_{G}(u,v) \leq d$. It is important to note that if $C_n$ is a cycle on $n$ vertices, then for $d \geq \lfloor \frac{n}{2} \rfloor$ the graph $C_n^{d}$ is the complete graph $K_n$. This holds true because $\text{dist}_{C_n}(u,v) \leq \lfloor \frac{n}{2} \rfloor$ for any two vertices $u$ and $v$ in $C_n$. It is known from \cite[Proposition 2.8]{Acharya1} that $\sigma^{-}(K_n) = \lfloor \frac{n}{2} \rfloor \lceil \frac{n}{2} \rceil$, so $\sigma^{-}(C_n^{d}) = \lfloor \frac{n}{2} \rfloor \lceil \frac{n}{2} \rceil$ for all $d \geq \lfloor \frac{n}{2} \rfloor$. Also it is known from \cite[Theorem 3.2]{Acharya2} that $\sigma^{-}(C_n^{1}) =2$. Thus it remains to compute the \emph{rna} number of $C_n^{d}$ for $2 \leq d < \lfloor \frac{n}{2} \rfloor$.

In this paper, we show for $2 \leq d < \lfloor \frac{n}{2} \rfloor$ that $2d \leq \sigma^{-}(C_n^{d}) \leq d(d+1)$ improving the upper bound of Kang et al. \cite{Kang} for the graph $C_n^{d}$. Moreover, we show that the \emph{rna} numbers of $C_n^{2}$ and $C_n^{3}$ are $6$ and $12$, respectively, achieving the upper bound of $d(d+1)$. For the remaining values of $d$, we make a conjecture that $\sigma^{-}(C_n^{d}) = d(d+1).$

\section{Preliminaries}\label{prelim}

We start this section with some necessary definitions.

\begin{definition}\cite{Acharya2}\label{def of parity labeling and parity signed graph}
\rm{For a given graph $G$ of order $n$ and a bijective mapping $f \colon V(G) \rightarrow \{1,\ldots ,n\}$, define $\sigma_{f} \colon E(G) \rightarrow \{1,-1\}$ such that $\sigma_{f}(uv)=1$ if $f(u)$ and $f(v)$ are of the same parity and $\sigma_{f}(uv)=-1$ if $f(u)$ and $f(v)$ are of different parity, where $uv \in E(G)$. We define $\Sigma_{f}$ to be the signed graph $(G,\sigma_{f})$, which is called a \textit{parity signed graph}.} 
\end{definition}

A cycle in a signed graph is said to be \textit{positive} if it has even number of negative edges, and \textit{negative}, otherwise. A signed graph is said to be \textit{balanced} if all of its cycles are positive. Every parity signed graph is balanced, see~\cite[Theorem 1]{Acharya2}.

\begin{definition}\cite{Acharya1}
\rm{The \emph{rna} number of a graph $G$, denoted $\sigma^{-}(G)$, is the least number of negative edges among all possible parity signed graphs over $G$.} 
\end{definition}

If $(G,\sigma_{f})$ is a parity signed graph, then the sets defined by $V_1 \coloneqq \{v \in V(G)~|~f(v)= \text{odd}\}$ and $V_2 \coloneqq \{v \in V(G)~|~f(v)=\text{even}\}$ are called \textit{parity sets}. Such a bipartition $\{V_1,V_2\}$ of $V(G)$ is said to be \textit{parity-partition} and an edge cut $(V_1,V_2)$ is called an \textit{equicut} of $G$. Observe that finding the \emph{rna} number of a graph $G$ is equivalent to finding the smallest equicut size.

Switching is a way of turning one signed graph into another signed graph. If $(G,\sigma)$ is a signed graph and $v$ is a vertex in $(G,\sigma)$, then \textit{switching} $v$ reverses the signs of all edges incident to $v$. Further, two signed graphs $(G,\sigma)$ and $(G,\sigma^\prime)$ are \textit{switching-equivalent} if one can be obtained from the other by a sequence of switchinges. 

For the study of parity signed graphs, a different type of switching is defined by Kang \textit{et al.} \cite{Kang}. Let $(G,\sigma)$ be a parity signed graph. A \textit{parity-switching} in $(G,\sigma)$ is a switching of a pair of vertices locating in different parity sets.

The following proposition tells us how we can obtain all parity signed graphs over any given graph $G$. 

\begin{prop}\cite{Kang}
A signed graph $(G, \sigma)$ on $n$ vertices is a parity signed graph if and only if it can be obtained from $(G, +)$ by switching a set of vertices of cardinality $\lfloor \frac{n}{2} \rfloor$.
\end{prop}

Kang \textit{et al.} also proved that parity-switching is the exact way to describe the transformation between any two parity signed graphs having the same underlying graph. More precisely:

\begin{prop}\cite{Kang}
Let $(G, \sigma)$ be a parity signed graph. Then a signed graph $(G, \sigma^\prime)$ is a parity signed graph if and only if it can be obtained from $(G, \sigma)$ by a sequence of parity-switchinges.
\end{prop}

The following result, due to Sehrawat and Bhattacharjya, gives a lower bound for the \emph{rna} number of a graph in terms of its edge-connectivity. 

\begin{prop}\cite{Sehrawat} \label{Prop-rna no and edge-connect}
If $G$ is a graph with edge-connectivity $\lambda$, then  $\sigma^{-}(G) \geq \lambda$.
\end{prop}

\section{Our results}\label{Sec: main results}
A simple but important result is the following.

\begin{lem}
Let $H$ be a spanning subgraph of a graph $G$ with $\sigma^{-}(H)=r$. Then $\sigma^{-}(G) \geq r$.
\end{lem} 
\begin{proof}
The size of each equicut of $G$ is at least the size of equicut of $H$. Thus the result follows.
\end{proof}

Let $S=\{a_1, \ldots, a_k\}$ be a set of positive integers such that $a_1 < a_2 < \cdots < a_k < \frac{p+1}{2}$ and let the vertices of a graph of order $p$ be labelled with $u_{0},u_{1},\ldots ,u_{p-1}$. Then the \textit{circulant graph} $C(p,S)$ has its vertex set as $\{u_{0},u_{1},\ldots ,u_{p-1}\}$ and vertices $u_{i \pm a_1}, u_{i \pm a_2}, \ldots, u_{i \pm a_k}$ are adjacent to each vertex $u_{i}$ (where summation in the indices are taken modulo $p$). In this definition, the sequence $(a_i)$ is called the \textit{jump sequence} and the $a_i$ are called \textit{jumps}. Further, if $a_k \neq \frac{p}{2}$, then $C(p,S)$ is a regular graph of degree $2k$ and if $a_k = \frac{p}{2}$, then $C(p,S)$ is a regular graph of degree $2k-1$. 

From here onwards, we assume that the vertex set of the cycle $C_n$ is $\{u_{0},u_{1},\ldots ,u_{n-1}\}$. Observe that for any $d$ such that $2 \leq d < \lfloor \frac{n}{2} \rfloor$ the graph $C_n^{d}$ is equivalent to the circulant graph $C(n,\{1,2,\ldots,d\})$. It is well-known that every circulant graph is vertex-transitive. Thus for each $2 \leq d < \lfloor \frac{n}{2} \rfloor$, the graph $C_n^{d}$ is also vertex-transitive. Mader~\cite{Mader} has been shown that for a connected vertex-transitive graph $G$, the edge-connectivity of $G$ is same as its minimum degree. We will use this fact to obtain a lower bound for the \emph{rna} number of $C_n^{d}$.

If $X$ is a set of $k+1$ consecutive vertices $u_{0},u_{1},\ldots ,u_{k}$ then the vertex (respectively, vertices) $u_{\frac{k}{2}}$ ($u_{\lfloor \frac{k}{2} \rfloor}~\&~u_{\lfloor \frac{k}{2} \rfloor +1}$) is (are) called \textit{mid-vertex (mid-vertices)} of $X$, according as $k$ is even (odd). If $v$ is a vertex in $G$ and $X$ is a subset of vertices of $G$ such that $v \notin X$, then for simplicity, we denote the edge-cut $E(\{v\}, X)$ and its size $|E(\{v\}, X)|$ by $(v, X)$ and $|(v, X)|$, respectively.

 Now we are ready to prove our main result.

\begin{theorem}\label{Thm-upper bound on rna no}
If $2 \leq d < \lfloor \frac{n}{2} \rfloor$, then $2d \leq \sigma^{-}(C_n^{d}) \leq d(d+1)$.
\end{theorem}
\begin{proof}
Since $C_n^{d}$ is a connected vertex-transitive graph with minimum degree $2d$, the edge-connectivity of $C_n^{d}$ is $2d$. Thus the lower bound follows by Proposition~\ref{Prop-rna no and edge-connect}.

For the rest of the proof, $X$ denotes the set $\{u_{0},u_{1},\ldots ,u_{\lfloor \frac{n}{2} \rfloor -1}\}$, where $n \geq 5$ and $|X| = \lfloor \frac{n}{2} \rfloor$. We distinguish two cases depending upon $\lfloor \frac{n}{2} \rfloor$ is odd or even.\\

\textit{\textbf{Case 1.}} Assume that $\lfloor \frac{n}{2} \rfloor = 2k+1$ for some $k$. Here $X = \{u_{0},u_{1},\ldots ,u_{2k}\}$ with $|X| =2k+1$, and the mid-vertex of $X$ is $u_{k}$. It is clearly observed that the graph $C_n^{d}$ is symmetric about the line passing through the vertex $u_{k}$. Consequently, the cardinality of two sets $(u_{k-i},X^{c})$ and $(u_{k+i},X^{c})$ is same for each $1 \leq i \leq k$. Therefore we conclude that 

\begin{equation}\label{equa-1-case:1}
 |(X,X^c)| =  2 \sum_{j=0}^{k-1} |(u_{j},X^c)| + |(u_{k},X^c)|.
\end{equation}

Now we find that
\begin{equation}\label{equa-2-case:1}
 |(u_{k},X^c)| = \left\{ \begin{array}{ll}
0 & \textrm{if $2 \leq d \leq k$,}\\
2\ell & \textrm{if $d= k+\ell$ and $\ell=1,\ldots , k$} \end{array}\right. 
\end{equation}

and for $0 \leq j \leq k-1$ that

\begin{equation}\label{equa-3-case:1}
 |(u_{j},X^c)| = \left\{ \begin{array}{ll}
d-j & \textrm{if $2 \leq d \leq k$ and $j \leq d$,}\\
0 & \textrm{if $2 \leq d \leq k$ and $j > d$,}\\
d-j & \textrm{if $d= k + \ell$, $\ell=1,\ldots ,k$ and $j+d \leq 2k$,}\\
2\ell & \textrm{if $d= k + \ell$, $\ell=1,\ldots ,k$ and $j+d \geq 2k+1$.} \end{array}\right.
\end{equation}

If $2 \leq d \leq k$, then from Equation~(\ref{equa-1-case:1}) we have 
\begin{align*}
|(X,X^c)| & = 2 \left[ \sum_{j=0}^{d} |(u_{j},X^c)| +  \sum_{j=d+1}^{k-1} |(u_{j},X^c)| \right] + 0 \\
 & =  2 \left[ \sum_{j=0}^{d} (d-j) + 0 \right] + 0\\
 & = d(d+1),
\end{align*}
where second equality uses Formulas (\ref{equa-2-case:1}) and (\ref{equa-3-case:1}).

If $d = k+\ell$ and $\ell = 1, \ldots , k$, then from Equation~(\ref{equa-1-case:1}) we have  
 
\begin{align*}
|(X,X^c)| & = 2 \left[ \sum_{j=0}^{2k-d} |(u_{j},X^c)| +  \sum_{2k-d+1}^{k-1} |(u_{j},X^c)| \right] + |(u_{k},X^c)|\\
 & = 2 \left[ \sum_{j=0}^{2k-d} (d-j) + \sum_{2k-d+1}^{k-1} (2 \ell) \right] + 2 \ell \\
 & =  2 \left[ d(2k-d+1) - \frac{(2k-d)(2k-d+1)}{2} \right] + (4d-4k) (d-k+1) + (2 d-2k)\\
 & = [4dk - 2d^2 +2d - 4k^2 + 4kd - d^2 - 2k+d] + [4d^2 -8kd + 4k^2 - 4d +4k] +(2d-2k)\\
 & = d(d+1),
\end{align*}
where second equality uses Formulas (\ref{equa-2-case:1}) and (\ref{equa-3-case:1}), and third equality uses the fact that $2 \ell = 2d-2k$. Thus for all $2 \leq d \leq 2k$, we conclude that $|(X,X^c)|=d(d+1).$

\textit{\textbf{Case 2.}} Assume that $\lfloor \frac{n}{2} \rfloor = 2k$ for some $k$. Here $X = \{u_{0},u_{1},\ldots ,u_{2k-1}\}$ with $|X| =2k$, and the mid-vertices of $X$ are $u_{k-1}$ and $u_{k}$. Further, the graph $C_n^{d}$ is symmetric about the line passing through the mid-point of $u_{k-1}$ and $u_{k}$. Consequently, the cardinality of two sets $(u_{(k-1)-j},X^{c})$ and $(u_{k+j},X^{c})$ is same for each $0 \leq j \leq k-1$. Therefore we conclude that 

\begin{equation}\label{equa-1-case:2}
 |(X,X^c)| =  2 \sum_{j=0}^{k-1} |(u_{j},X^c)|.
\end{equation}

Now we find for $0 \leq j \leq k-1$ that 

\begin{equation}\label{equa-2-case:2}
 |(u_{j},X^c)| = \left\{ \begin{array}{ll}
0 & \textrm{if $2 \leq d \leq k-1$ and $j \geq d$}\\
d-j & \textrm{if $2 \leq d \leq k-1$ and $j < d$}  \\

d-j & \textrm{if $d= (k-1)+\ell$, $\ell=1,\ldots ,k$ and $j+d \leq 2k-1$}\\

2\ell-1 & \textrm{if $d= (k-1)+\ell$, $\ell=1,\ldots ,k$ and $j+d \geq 2k$.} \end{array}\right. 
\end{equation}

If $2 \leq d \leq k-1$, then from Equation~(\ref{equa-1-case:2}) we have 
\begin{align*}
|(X,X^c)| & = 2 \sum_{j=0}^{k-1} |(u_{j},X^c)|  \\
 & = 2 \sum_{j=0}^{d-1} |(u_{j},X^c)| + 2 \sum_{j=d}^{k-1} |(u_{j},X^c)| \\
 & = 2 \sum_{j=0}^{d-1} (d-j) + 0 + 0 \\
 & = d(d+1).
\end{align*}

If $d = (k-1) + \ell$ and $\ell=1,\ldots ,k$, then from Equation~(\ref{equa-1-case:2}) we have 
\begin{align*}
|(X,X^c)| & = 2 \sum_{j=0}^{k-1} |(u_{j},X^c)|  \\
 & = 2 \sum_{j=0}^{2k-d-1} |(u_{j},X^c)| + 2 \sum_{2k-d}^{k-1} |(u_{j},X^c)| \\
 & = 2 \sum_{j=0}^{2k-d-1} (d-j) + 2 \sum_{2k-d}^{k-1} (2 \ell -1) \\
 & = d(d+1).
\end{align*}

Hence for all $2 \leq d \leq 2k-1$, we have $|(X,X^c)| =d(d+1)$.

In both the cases, we have shown that the size of $(X,X^c)$ is $d(d+1)$ for $X=\{u_{0},u_{1},\ldots ,u_{\lfloor \frac{n}{2} \rfloor -1}\}$. Thus we get, $\sigma^{-}(C_n^{d}) \leq d(d+1)$. This completes the proof.
\end{proof}

\subsection{The \emph{rna} number of $C_n^{2}$}

\begin{theorem}\label{Thm: rna no of C_n^2}
If $n \geq 6$, then $\sigma^{-}(C_n^2) = 6$.
\end{theorem}
\begin{proof} We discuss the two cases depending upon the parity of $n$.

\textit{\textbf{Case 1.}} Assume that $n$ is even. In this case, $C_n^2$ is edge-disjoint union of one $n$-cycle (formed with $1$ jumps) and two $\frac{n}{2}$-cycles (formed with $2$ jumps). We denote the $n$-cycle by $C_n^{\prime} \coloneqq u_0u_1 \ldots u_{n-1}u_0$ and two $\frac{n}{2}$-cycles by $C_{\frac{n}{2}}^{\prime} \coloneqq u_1u_3 \ldots u_{n-1}u_1$ and $C_{\frac{n}{2}}^{\prime\prime} \coloneqq u_0u_2 \ldots u_{n-2}u_0$. So for any $X \subset V(C_n^2)$ such that $|X| = \frac{n}{2}$, we discuss the following two cases.

\textit{\textbf{Case 1.1:}} Let us consider $X$ contain vertices of both $C_{\frac{n}{2}}^{\prime}$ as well as $C_{\frac{n}{2}}^{\prime\prime}$. Thus the cut $(X,X^c)$ must contain at least two edges of each of the cycles $C_{\frac{n}{2}}^{\prime}$, $C_{\frac{n}{2}}^{\prime\prime}$, and $C_n^{\prime}$. Therefore, $|(X,X^c)| \geq 6$.

\textit{\textbf{Case 1.2:}} Let us consider $X$ contain all the vertices of exactly one of $C_{\frac{n}{2}}^{\prime}$ and $C_{\frac{n}{2}}^{\prime\prime}$. Thus $(X,X^c) = E(C_n)$. Consequently, $|(X,X^c)| = n \geq 6$. 

\textit{\textbf{Case 2:}} Assume that $n$ is odd. In this case, $C_n^2$ is edge-disjoint union of two $n$-cycles (one formed with $1$ jumps while other formed with $2$ jumps). We denote these two $n$-cycles by $C_{n}^{\prime}$ and $C_{n}^{\prime\prime}$, where $C_{n}^{\prime}$ and $C_{n}^{\prime\prime}$ are formed with $1$ jumps and $2$ jumps, respectively. That is, $C_{n}^{\prime} = u_0u_1 \ldots u_{n-1}u_0$ and $C_{n}^{\prime\prime} = u_0u_2 \ldots u_{n-1}u_1u_{3} \ldots u_{n-2}u_0$. 

Let $X$ be a subset of $V(C_n^2)$ such that $|X| = \lfloor \frac{n}{2} \rfloor$. Clearly, $(X,X^c)$ must contain at least two edges of each of $C_{n}^{\prime}$ and $C_{n}^{\prime\prime}$. Further, if $(X,X^c)$ contain exactly two edges of $C_{n}^{\prime}$ then $X$ must be of the form $\{u_{j}, u_{j+1}, \ldots , u_{j+ \lfloor \frac{n}{2} \rfloor -1}\}$ for some $j \in \{0,1, \ldots , n-1\}$. Consequently, $X^c = V(C_n) - X$. Thus the following edges: $$u_{j-2}u_{j}, u_{j-1}u_{j+1}, u_{j+ \lfloor \frac{n}{2} \rfloor-2}u_{j+ \lfloor \frac{n}{2} \rfloor}, u_{j+ \lfloor \frac{n}{2} \rfloor -1}u_{j+ \lfloor \frac{n}{2} \rfloor + 1}$$
of $C_{n}^{\prime\prime}$, along with two edges of $C_{n}^{\prime}$, have to belong to $(X,X^c)$. Hence $|(X,X^c)| \geq 6$.

On the other hand, if $(X,X^c)$ contain more than two edges of $C_{n}^{\prime}$, then $(X,X^c)$ has to contain at least 4 edges of $C_{n}^{\prime}$ and hence $|(X,X^c)| \geq 6$. 

Thus from Cases 1 and 2, we conclude that any equicut of $C_n^2$ has to be of size at least 6. Therefore, $\sigma^{-}(C_n^2) \geq 6$. Now consider $X=\{u_0, \ldots, u_{\lfloor \frac{n}{2} \rfloor -1}\}$. So we have $$(X,X^c) = \{u_{n-1}u_{0}, u_{n-1}u_{1}, u_{n-2}u_{0}, u_{\lfloor \frac{n}{2} \rfloor -2}u_{\lfloor \frac{n}{2} \rfloor}, u_{\lfloor \frac{n}{2} \rfloor-1}u_{\lfloor \frac{n}{2} \rfloor}, u_{\lfloor \frac{n}{2} \rfloor-1}u_{\lfloor \frac{n}{2} \rfloor +1}\}.$$ It is clear that $(X,X^c)$ is an equicut of $C_n^2$ and that $|(X,X^c)| = 6$. Hence $\sigma^{-}(C_n^2) = 6$, as desired.
\end{proof}

\subsection{The \emph{rna} number of $C_n^{3}$}
\begin{theorem}\label{Thm: rna no of C_n^3}
If $n \geq 8$, then $\sigma^{-}(C_n^3) = 12$.
\end{theorem}
\begin{proof} 
Let $X$ be a subset of $V(C_n^3)$ such that $|X|=\lfloor \frac{n}{2} \rfloor$. The number of cycles of different lengths (formed with 2 jumps and 3 jumps) in $C_n^3$ will depend upon the values of $n$. So, we discuss the following cases. 

\vspace{0.1in}
\noindent
\textit{\textbf{Case 1:}} Let $n$ be an even and a multiple of 3. In this case, $n \geq 12$ and $C_n^3$ is edge-disjoint union of one $n$-cycle (formed with 1 jumps), two $\frac{n}{2}$-cycles (formed with 2 jumps) and three $\frac{n}{3}$-cycles (formed with 3 jumps). We denote these $\frac{n}{2}$-cycles by $C_{\frac{n}{2}}^{1}$ and $C_{\frac{n}{2}}^{2}$, and $\frac{n}{3}$-cycles by $C_{\frac{n}{3}}^{1}$, $C_{\frac{n}{3}}^{2}$ and $C_{\frac{n}{3}}^{3}$. It is important to note that for any $i \in \{1,2,3\}$, each vertex of $V(C_n) \setminus V(C_{\frac{n}{3}}^{i})$ has exactly two neighbours in $V(C_{\frac{n}{3}}^{i})$. So if $V(C_{\frac{n}{3}}^{i}) \subset X$, then each vertex of $X^c$ contributes at least two edges in $(X,X^c)$. Hence $|(X,X^c)| \geq 2 \times \lceil \frac{n}{2} \rceil \geq n$. Thus $|(X,X^c)| \geq 12$. 

If $X=V(C_{\frac{n}{2}}^{i})$ for any $i \in \{1,2\}$, then $E(C_n) \subseteq (X,X^c)$. Thus $|(X,X^c)| \geq n \geq 12$. 

Next, if $X$ contain vertices of $C_{\frac{n}{2}}^{i}$, and $C_{\frac{n}{3}}^{j}$, for each $i \in \{1,2\}$ and $j \in \{1,2,3\}$, then $X$ has to contain vertices of all 6 cycles. Consequently, $|(X,X^c)| \geq 12$.

\vspace{0.1in}
\noindent
\textit{\textbf{Case 2:}} Let $n$ be an even and not a multiple of 3. That is, $n=2k$ for some $k \geq 4$ and $k$ is not divisible by 3. In this case, $n \geq 8$. For any $X \subset V(C_{n})$ such that $|X|=k$ there are only two choices: (i) $X$ contains all the vertices of exactly one of the cycles (formed with 2 jumps) of length $k$, or (ii) $X$ contains vertices of both the cycles of length $k$. We discuss these two cases separately. 

\textit{\textbf{Case 2.1:}} Without loss of generality, we assume that $X= \{u_{1},u_{3}, \ldots , u_{2k-1}\}$. That is, $X$ contains all the vertices of exactly one of the cycles of length $k$. Note that $$(X,X^c) = \{u_{i}u_{i+1}, u_{i}u_{i+3}~|~i\in \{0,1, \ldots , 2k-1\}\}.$$ Consequently, $|(X,X^c)| = 2n \geq 16$ because $n \geq 8$.

\textit{\textbf{Case 2.2:}} If $X$ contains vertices of both the cycles of length $k$, then $(X,X^c)$ contains at least 8 edges. This holds true because in this case, there are total four cycles in $C_n^3$ and vertices of each of these four cycles belong to both $X$ and $X^c$. 

Now it is important to observe that any cut $(X,X^c)$ of a cycle on $n$ vertices with $|(X,X^c)|=2$ must correspond to a vertex set $X = \{u_{i},u_{i+1}, \ldots , u_{i+ \lfloor \frac{n}{2} \rfloor -1}\}$ for some $i \in \{0,1, \ldots , n-1\}$. Therefore, if $(X,X^c)$ contains exactly two edges of an $n$-cycle (formed with 1 jumps), then $X$ must equals $\{u_{i},u_{i+1}, \ldots , u_{i+ \lfloor \frac{n}{2} \rfloor -1}\}$. Hence $(X,X^c)$ has to contain the edges $u_{i-2}u_{i},u_{i-1}u_{i+1},u_{i+ \lfloor \frac{n}{2} \rfloor -1}u_{i+ \lfloor \frac{n}{2} \rfloor +1}, \linebreak[4] u_{i+ \lfloor \frac{n}{2} \rfloor -2}u_{i+ \lfloor \frac{n}{2} \rfloor},u_{i-3}u_{i}$, $u_{i-2}u_{i+1},u_{i-1}u_{i+2},u_{i+ \lfloor \frac{n}{2} \rfloor -3}u_{i+ \lfloor \frac{n}{2} \rfloor},u_{i+ \lfloor \frac{n}{2} \rfloor -2}u_{i+ \lfloor \frac{n}{2} \rfloor +1},u_{i+ \lfloor \frac{n}{2} \rfloor -1}u_{i+ \lfloor \frac{n}{2} \rfloor +2}$ along with those two edges of $n$-cycle. Thus, $|(X,X^c)| \geq 12$. 

Similarly, if $(X,X^c)$ contains exactly two edges of an $n$-cycle (formed with 3 jumps), then $|(X,X^c)| \geq 12$. Thus, we conclude that if $(X,X^c)$ contains exactly two edges of any one of $n$-cycles (formed with 1 jumps or 3 jumps), then $|(X,X^c)| \geq 12$.

On the other hand, if $(X,X^c)$ contains at least 4 edges of both $n$-cycles, then $|(X,X^c)| \geq 12$ because both $k$-cycles contribute at least 4 edges in $(X,X^c)$.

\vspace{0.1in}
\noindent
\textit{\textbf{Case 3:}} Let $n$ be an odd and divisible by 3. That is, $n=2k+1$ for some $k \geq 4$ and is divisible by 3. Here $C_{n}^3$ has total 5 cycles. Out of these 5 cycles, two are $n$-cycles (formed with 1 jumps and 2 jumps) and three are $\frac{n}{3}$-cycles. Let $X$ be any subset of $V(C_n)$ such that $|X|=k$. As discussed in Case 1, if $X$ contains all the vertices of any one of $\frac{n}{3}$-cycles, then $|(X,X^c)| \geq 12$. 

On the other hand, if $X$ contains vertices of each of 5 cycles, then it is clear that $|(X,X^c)| \geq 10$. Moreover this equality holds only if $(X,X^c)$ contains exactly two edges of each of these 5 cycles. Now we distinguish the following cases.

\textit{\textbf{Case 3.1:}} If $(X,X^c)$ contains exactly two edges of the cycle $u_{0}u_{1}\ldots u_{n-1}u_{0}$, then $X$ must equals $\{u_{i},u_{i+1}, \ldots , u_{i+k-1}\}$ for some $i \in \{0,1,\ldots , n-1\}$. Consequently, $(X,X^c)= \{u_{i-1}u_{i},u_{i-1}u_{i+1},u_{i-1}u_{i+2}, \linebreak[4] u_{i-2}u_{i},u_{i-2}u_{i+1},u_{i-3}u_{i}, u_{i+k-3}u_{i+k},u_{i+k-2}u_{i+k},u_{i+k-2}u_{i+k+1},u_{i+k-1}u_{i+k},u_{i+k-1}u_{i+k+1},u_{i+k-1}u_{i+k+2}\}.$ Hence $|(X,X^c)| = 12$.

\textit{\textbf{Case 3.2:}} If $(X,X^c)$ contains at least four edges of the cycle $u_{0}u_{1}\ldots u_{n-1}u_{0}$, then $|(X,X^c)| \geq 12$. Therefore, we conclude that $|(X,X^c)| \geq 12$.   

\vspace{0.1in}
\noindent
\textit{\textbf{Case 4:}} Let $n$ be an odd and not divisible by 3. Note that if $n=2k+1$ is odd and not divisible by 3, then $n$ must be of the form $3 \ell -1$ or $3 \ell-2$, according as $\ell$ is even or odd, respectively. 

In this case, $C_{n}^{3}$ has total 3 $n$-cycles. We denote these cycles by $C_{n1}$, $C_{n2}$, and $C_{n3}$, where $C_{ni}$ is formed with $i$ jumps for $i=1,2,3$. Clearly for any $X \subset V(C_n)$ with $|X| = \lfloor \frac{n}{2} \rfloor$, the cut $(X,X^c)$ contains at least 2 edges of each of $C_{ni}$ for $i=1,2,3$. Thus $|(X,X^c)| \geq 6$. But we claim that for each $X \subset V(C_n)$ with $|X| = \lfloor \frac{n}{2} \rfloor$, the cut $(X,X^c)$ contains at least 4 edges of each of $C_{ni}$ for $i=1,2,3$. This would prove that $|(X,X^c)| \geq 12$. To justify this claim, we discuss the following cases.

\textit{\textbf{Case 4.1:}} If $(X,X^c)$ contains exactly 2 edges of $C_{n1}$, then as discussed in Case 3.1, we have $|(X,X^c)| = 12$, as desired.

\textit{\textbf{Case 4.2:}} If $(X,X^c)$ contains exactly 2 edges of $C_{n2}$, then $X$ must equals $\{u_{i},u_{i+2}, \ldots ,u_{i+2k-2}\}$ for some $i \in \{0,1, \ldots , n-1\}$. So we have $(X,X^c)= E(C_{n1}) \cup \{u_{i-2}u_{i}, u_{i+2k-2}u_{i+2k}\} \cup S - \{u_{i+(2k-1)}u_{i+2k}\}$, where $S \subset E(C_{n3})$ and $|S| \geq 0$. Thus $|(X,X^c)| \geq n+2-1 = n+1 \geq 12$, because $n \geq 11$.

\textit{\textbf{Case 4.3:}} Let $(X,X^c)$ contains exactly 2 edges of $C_{n3}$. If $n = 3 \ell -1$ and $\ell \geq 4$ is even, then $$X = \{u_{i},u_{i+3}, \ldots ,u_{i+3(\ell -1)},u_{i+1}, \ldots, u_{i+3(\frac{\ell}{2} -1)-2}\}$$ so that $|X|= \ell +\frac{\ell}{2} -1 = \frac{3 \ell}{2} -1$. Consider a set $$A = 
\begin{cases} 
\{u_{i+2},u_{i+5},u_{i+7},u_{i+8},u_{i+10}\}~~ \text{if $\ell = 4$}\\
\{u_{i+2},u_{i+5},u_{i+8},u_{i+3\ell-4},u_{i+3\ell-2}\}~~ \text{if $\ell \geq 6$ is even}.\\
\end{cases}$$

Observe that $A \subset X^c$. Hence, for $\ell =4$, the edges $u_{i}u_{i+2},u_{i}u_{i+8},u_{i}u_{i+10},u_{i+3}u_{i+5},u_{i+3}u_{i+2},u_{i+6}u_{i+5}, \linebreak[4] u_{i+6}u_{i+7},u_{i+6}u_{i+8},u_{i+9}u_{i+8},u_{i+9}u_{i+7},u_{i+9}u_{i+10},u_{i+1}u_{i+2},u_{i+1}u_{i+10}$ must belong to the cut $(X,X^c)$. Thus $|(X,X^c)| > 12$. For $\ell \geq 6$, the edges $u_{i}u_{i+2},u_{i}u_{i+3\ell -2},u_{i}u_{i+3\ell -4},u_{i+3}u_{i+2},u_{i+3}u_{i+5},u_{i+6}u_{i+5}, \linebreak[4] u_{i+6}u_{i+8},u_{i+9}u_{i+8}, u_{i+3\ell -3}u_{i+3\ell -4},u_{i+3\ell -3}u_{i+3\ell -2},u_{i+1}u_{i+2},u_{i+4}u_{i+2},u_{i+4}u_{i+5}$ must belong to $(X,X^c)$. Thus $|(X,X^c)| > 12$.

Now if $n = 3 \ell -2$ and $\ell \geq 5$ is odd, then $$X = \{u_{i},u_{i+3}, \ldots ,u_{i+3(\ell -1)},u_{i+2}, \ldots, u_{i+3(\lfloor \frac{\ell}{2} \rfloor -1)-1}\}$$ so that $|X|= \ell +\lfloor \frac{\ell}{2} \rfloor -1 = \lfloor \frac{3 \ell}{2} \rfloor -1$. Here $X^c$ must contain the vertices $u_{i+1},u_{i+4},u_{i+7},u_{i+3 \ell-5},$ and $u_{i+3 \ell -4}$. So the edges $u_{i}u_{i+1},$ $u_{i}u_{i+3 \ell-4},$ $u_{i+3}u_{i+1},$ $u_{i+3}u_{i+4},$ $u_{i+6}u_{i+4},$ $u_{i+6}u_{i+7},$ $u_{i+9}u_{i+7},$ $u_{i+3 \ell-6}u_{i+3 \ell-5}, \linebreak[4] u_{i+3 \ell-6}u_{i+3 \ell-4},u_{i+3 \ell-3}u_{i+3 \ell-5},u_{i+3 \ell-3}u_{i+3 \ell-4},u_{i+3 \ell-3}u_{i+1}$ must belong to $(X,X^c)$. Thus $|(X,X^c)| > 12$.

From all the above four cases, we conclude that for any $X \subset V(C_{n}^{3})$ with $|X| = \lfloor \frac{n}{2} \rfloor$, we have $|(X,X^c)| \geq 12$. Further, if $X = \{u_0, u_1, \ldots, u_{\lfloor \frac{n}{2} \rfloor -1}\}$, then $(X,X^c) = \{u_{n-1}u_{0},u_{n-1}u_{1},u_{n-1}u_{2}, u_{n-2}u_{0}, \linebreak[4] u_{n-2}u_{1}, u_{n-3}u_{0},u_{\lfloor \frac{n}{2} \rfloor -3}u_{\lfloor \frac{n}{2} \rfloor},u_{\lfloor \frac{n}{2} \rfloor -2}u_{\lfloor \frac{n}{2} \rfloor },u_{\lfloor \frac{n}{2} \rfloor -2}u_{\lfloor \frac{n}{2} \rfloor +1},u_{\lfloor \frac{n}{2} \rfloor -1}u_{\lfloor \frac{n}{2} \rfloor},u_{\lfloor \frac{n}{2} \rfloor -1}u_{\lfloor \frac{n}{2} \rfloor +1},u_{\lfloor \frac{n}{2} \rfloor-1}u_{\lfloor \frac{n}{2} \rfloor +2}\}$. So we have $|(X,X^c)| = 12$. Consequently, $\sigma^{-}(C_n^3) = 12$. This completes the proof.
\end{proof}

In this paper, we obtained an upper bound of $d(d+1)$ for the \emph{rna} number of $C_n^{d}$ for $2 \leq d < \lfloor \frac{n}{2} \rfloor$ and proved that the \emph{rna} numbers of $C_n^{2}$ and $C_n^{3}$ are 6 and 12, respectively. We were not able to calculate the \emph{rna} number of $C_n^{d}$ for $4 \leq d < \lfloor \frac{n}{2} \rfloor$. So for these values of $d$ we make the following conjecture.

\begin{conjecture}
If $4 \leq d < \lfloor \frac{n}{2} \rfloor$, then $\sigma^{-}(C_n^{d}) = d(d+1).$
\end{conjecture}
 
%\section{Declarations} 
%
%\textbf{Conflict of interests}
%The authors have no conflict of interests to disclose.
%
%\vspace{0.01in}
%\noindent
%\textbf{Data availability}
%The authors confirm that no data were used during the current study.

\end{document}